\documentclass[12pt,reqno]{amsart}
\usepackage[english]{babel}
\usepackage{lineno}
\usepackage{lmodern}
\usepackage[autopunct=true]{csquotes}
\DeclareQuoteStyle[american]{english}% verified
  {\textquotedblleft\kern0.01em}% force additional kerning
{\textquotedblright}
  [0.05em]% Fallback only for fonts without kerning
  {\textquoteleft}
  {\textquoteright}
\usepackage{graphicx}
\usepackage{psfrag}
\usepackage{pxfonts}
\usepackage{mathrsfs}
\usepackage[dvipsnames]{xcolor}
\usepackage{amsmath,amssymb}
\usepackage{soul}
\usepackage{cancel}
\newcommand{\bp}{\it Proof.}

\newcommand{\ep}[1]{\begin{flushright}\fbox{\hspace{.1em}}\:[#1]\end{flushright}}

\makeatletter
\@namedef{subjclassname@2020}{%
  \textup{2020} Mathematics Subject Classification}
\makeatother

\numberwithin{equation}{section}
\newcommand{\diff}{\operatorname{Diff}}
\newcommand{\PH}{\operatorname{PH}}

\newcommand{\eps}{\varepsilon}

\newcommand{\supp}{\operatorname{supp}}
\newcommand{\Fr}{\operatorname{\mathcal{F}\hspace{-.3em}r}}
\renewcommand{\div}{\operatorname{div}}
\DeclareMathOperator{\sgn}{sgn}
\newcommand{\Per}{{\rm Per}}

\def \cD {{\mathcal D}}
\def \cDA {{\rm{ DA}}}

\def \cP {{\mathcal P}}
\def \cR {{\mathcal R}}

\def \cW {{\mathcal W}}

\theoremstyle{plain}
\newtheorem{maintheorem}{Theorem}
\newtheorem{maintheo}{Main Theorem}

\newcommand{\R}{\mathbb{R}}
\newcommand{\N}{\mathbb{N}}

\newcommand{\T}{\mathbb{T}}
\newcommand{\F}{\mathcal{F}}

\newtheorem{theorem}{Theorem}[section]

\newtheorem{conjecture}[theorem]{Conjecture}
\newtheorem{cor}[theorem]{Corollary}
\newtheorem{proposition}[theorem]{Proposition}
\newtheorem{lemma}[theorem]{Lemma}
\newtheorem{definition}[theorem]{Definition}

\newtheorem{question}[theorem]{Question}
\newtheorem{claim}[theorem]{Claim}
\newtheorem{remark}[theorem]{{Remark}}

\newtheorem{notation}[theorem]{Notation}
\newtheorem{exercise}[theorem]{Exercise}

\def\bp{\noindent{\it Proof. }}

\def\bp{\noindent{\it Proof. }}

\begin{document}

\thanks{FPM was partially supported by NNSFC 12071202, JRH  was partially supported by NSFC 12161141002 and NSFC 12250710130, RM  was partially supported by NSFC 12250710130, and RU was partially supported by  NNSFC 12071202, and NNSFC
12161141002. }

\author[F. Micena]{Fernando Micena}
\address{F. Micena, Instituto de Matem\'{a}tica e Computa\c{c}\~{a}o,
  IMC-UNIFEI, Itajub\'{a}-MG, Brazil.}
\email{fpmicena82@unifei.edu.br}

\author{Ryo Moore}
\address{R. Moore,  
Department of Mathematics, Southern University of Science and Technology of China,
No. 1088, Xueyuan Rd., Xili, Nanshan District, Shenzhen,
Guangdong 518055, China.}
\email{ryomoore@sustech.edu.cn}

\author[J. Rodriguez Hertz]{Jana Rodriguez Hertz}

\address{J. Rodriguez Hertz,
Department of Mathematics of SUSTech and Shenzhen International Center for Mathematics, Southern University of Science and Technology of China,
No. 1088, Xueyuan Rd., Xili, Nanshan District, Shenzhen,
Guangdong 518055, China.}
\email{rhertz@sustech.edu.cn}

\author[R. Ures]{Raul Ures}
\address{R. Ures,
Department of Mathematics SUSTech and Shenzhen International Center for Mathematics, Southern University of Science and Technology of China,
No. 1088, Xueyuan Rd., Xili, Nanshan District, Shenzhen,
Guangdong 518055, China.}
\email{ures@sustech.edu.cn}

\subjclass[2020]{37A35; 37C15; 37C40; 37D25; 37D30}

\renewcommand{\subjclassname}{\textup{2020} Mathematics Subject Classification}

\date{\today}

\setcounter{tocdepth}{2}

\title{ Measures of maximal entropy that are SRB }
 \begin{abstract} 
A smooth conservative DA-diffeomorphism is smoothly conjugated to its Anosov linear part if and only if all Lyapunov exponents coincide almost everywhere with those of its linear part. \par\vspace{.3em}
A more general result for entropy maximizing measures of $C^{1+\alpha}$ partially hyperbolic diffeomorphisms isotopic to Anosov (DA-diffeomorphisms) on ${\mathbb T}^{3}$ is that they are SRB measures if and only if the sum of its positive Lyapunov exponents coincides with that of the linear Anosov map $A$ on all periodic orbits of the support of the measure. In that case, the measure is also the unique physical measure.\par
\vspace{.3em}

This rigidity result is not as strong as in the A. Katok rigidity conjecture. Examples are provided.
\end{abstract}
\maketitle

%----------------------------------------------------------
\section{Introduction}\label{introduction}
%----------------------------------------------------------

The {\bf metric entropy} of the volume probability measure $m$ with respect to a volume preserving diffeomorphism $f$ on a compact manifold is given by the {\bf Pesin's formula}:
\begin{equation}\label{pesin.formula}
h_{m}(f)=\int\sum_{\lambda_{i}(f,x)>0}\lambda_{i}(f,x)dm
\end{equation}
where $\lambda_{i}(f)$ are the positive Lyapunov exponents of $f$ counted with their multiplicity. More information in Section \ref{preliminaries}.\par
The complexity of $f$ can also be measured by its {\bf topological entropy}, $h_{top}(f)$. 
The classical Variational Principle states that
\begin{equation}\label{variational.principle}
h_{top} (f)=\sup\{h_{\mu}(f):= \mu\:\text{$f$-invariant probability measure }\}
\end{equation}
under certain circumstances.\par
Any measure attaining the maximum in equation \eqref{variational.principle} is a {\bf measure of maximal entropy} ({\bf m.m.e}). \par
How often does the complexity of a system with respect to a physically meaningful measure coincide with its intrinsic complexity?  
\begin{question}\label{main.question}
How often is the measure of maximal entropy an SRB measure? 
\end{question}

This problem was first addressed by A. Katok:
\begin{conjecture}[A. Katok, 1982 \cite{katok82}] \label{conj.katok}If $\phi_{t}$ is a geodesic flow of a manifold $M$ with negative curvature, then the Liouville measure is a measure of maximal entropy if and only if $M$ is locally symmetric. 
\end{conjecture}
Many people have worked on this conjecture, notably A. Katok himself, who proved it true in the two-dimensional case, Hamenst\"adt \cite{UH95}, Besson-Courtois-Gallot \cite{BCG95}, Foulon \cite{F01}, de Simoi - Leguil - Vinhage - Yang \cite{DLVY20}, and others.  (The authors thank F. Rodriguez Hertz for bringing this conjecture to their attention.) \par\vspace{.3em}

{\bf $\PH(\T^{3})$} is the set of partially hyperbolic diffeomorphisms in $\T^{3}$. For a hyperbolic linear automorphism $A$ of $\T^{3}$, the set 
\begin{equation}\label{DAdef}
{\mathbf \cDA(A)}=\{f\in\PH(\T^{3}): \exists\alpha:[0,1]\to \diff(\T^{3})\text{ such that }\alpha(0)=A,\, \alpha(1)=f\} 
\end{equation}
is the set {\bf DA-diffeomorphisms} of $A$ on $\T^{3}$. A {\bf DA-diffeomorphism} is an $f\in\cDA(A)$ for some $A$. 
\begin{notation}
    $$\cDA^{1+\alpha}(A)=\cDA(A)\cap \diff^{1+\alpha}(\T^3)\qquad{and}\qquad \cDA^{1+\alpha}=\cDA\cap \diff^{1+\alpha}(\T^3).$$
\end{notation}
$A$ is the linear part of all $f\in\cDA(A)$ and has real spectrum (see Section \ref{preliminaries}). 
$\lambda^{s}_{A}, \lambda_{A}^{c}$, and $\lambda^{u}_{A}$ are the {\bf stable, center, and unstable Lyapunov exponents of $A$}.\par\vspace{.3em}

\begin{notation}
For $f\in \cDA(A)$ and $p\in\Per_{H}(f)$ (the set of $f${\bf-hyperbolic periodic points} $p$ of $f$),
\begin{equation}\label{lambda+}
 {\mathbf{\lambda^{+}(p)}}=\left\{
\begin{array}{ll}
 \lambda_{f}^{u}(p)&\text{ if }\lambda^{c}_{A}<0\\
 \lambda_{f}^{u}(p)+\lambda_{f}^{c}(p)&\text{ if }\lambda^{c}_{A}>0.
\end{array}
 \right.
\end{equation}
$\lambda_{f}^{u}(p)$ and $\lambda_{f}^{c}(p)$ are the {\bf unstable and center Lyapunov exponents} of $p$ with respect to $f$.\par
\end{notation}
\setcounter{maintheo}{-1}
\begin{maintheo}[Lyapunov exponents rigidity] \label{thmC} \label{corolario.saghin.yang}
If $f:\T^{3}\to\T^{3}$ is a $C^{\infty}$ volume preserving DA-diffeomorphism, then $f$ is smoothly conjugated to its linear part $f_{*}$ if and only if 
\begin{equation}\label{lyapunov.rigidity}
 \lambda^{\sigma}(f)=\lambda^{\sigma}_{f_{*}}, \qquad\text{for }\sigma=s,c,u.
\end{equation}
\end{maintheo}
\setcounter{maintheo}{-1}
\vspace*{1em}
\subsection*{Organization of the paper}

The Main Theorem is proved in the last section. \par
Theorem \ref{thmA} (Section \ref{section.thm.A}) provides a complete answer to Question \ref{main.question} in the particular case of DA-diffeomorphisms of $\T^{3}$. 
This question is relevant in more general contexts: different manifolds, higher dimensions, and more general dynamics. 
However, for general partially hyperbolic diffeomorphisms in three-dimensional manifolds, this issue is open and has not been addressed yet. We consider this paper an analysis of a toy example that can provide hints for future generalizations.  \par\vspace{.3em}
Condition \eqref{condition.thm.A} in Theorem \ref{thmA} is very fragile in the sense that it can be easily broken by a $C^{\infty}$-perturbation, both in the conservative and non-conservative case, see Theorem \ref{thm.examples}. On the other hand, the set of diffeomorphisms $f\in\cDA$ for which condition \eqref{condition.thm.A} fails is a $C^{1}$-open set. \par\vspace{.3em}

 As stated in the Introduction, this article is inspired by Conjecture \ref{conj.katok}, which states that the entropy-maximizing measure of the geodesic flow of a manifold with negative curvature coincides with the Liouville measure if and only if the manifold is locally symmetric (constant curvature in the 2-dimensional manifold case). The conjecture was proved true for surfaces, but it remains open in the general case despite successive advances. \par\vspace{.3em} 
 We aim to obtain a result of rigidity in the spirit of Conjecture \ref{conj.katok}. This is the content of Theorem \ref{thmA}. Theorem \ref{thm.examples}  shows that this rigidity is not as ``rigid'' as expected in A. Katok's setting (in fact, people in the geodesic flows area were more surprised by Theorem \ref{thm.examples} than by Theorem \ref{thmA}). One can have a diffeomorphism for which the entropy-maximizing measure is SRB, or even Lebesgue, but the diffeomorphism is not a linear Anosov diffeomorphism - the analogous of ``constant curvature''-. \par\vspace{.3em}

From Theorems \ref{thmA} and \ref{thm.examples} it follows that, even though there is some degree of rigidity, this rigidity is not enough to guarantee $f$ to be an Anosov diffeomorphism. This panorama may change if we increase our requirements.
\begin{remark}\label{obs.da.ergodicos}
The Lyapunov exponents of $f$ for the volume probability measure are constant since all $f \in \cDA$ are ergodic \cite{ganshi}.  
\end{remark}

\begin{notation}\label{notation.lyap.exp.erg}
 In the case of a general $f$-invariant ergodic probability measure $\mu$, we will call its Lyapunov exponents $\lambda^{s}_{\mu}(f), \lambda^{c}_{\mu}(f)$, and $\lambda^{u}_{\mu}(f)$. 
\end{notation}
\subsection*{Acknowledgments:} The authors thank Pablo Carrasco for his support. The authors also acknowledge anonymous reviewers for their feedback that helped improve this presentation.
\section{Preliminaries}\label{preliminaries}
Let $M$ be a closed Riemannian manifold. A diffeomorphism $f:M\to M$ is {\bf partially hyperbolic} if the tangent bundle of $M$ admits a $Df$-invariant splitting $TM=E^{s}\oplus E^{c}\oplus E^{u}$ (the three subbundles are nontrivial), and a Riemannian metric such that all unit vectors $v^{\sigma}\in E^{\sigma}$, $\sigma=s,c,u$ satisfy 
$$\|Df(x)v^{s}\|<\|Df(x)v^{c}\|<\|Df(x)v^{u}\|,$$
$$\|Df(x)v^{s}\|<1<\|Df(x)v^{u}\|.$$
From now on, we consider $f\in \cDA(A)$, as defined in \eqref{DAdef}. The matrix $A$ has a real spectrum; see \cite[Corollary 5.1.10, Theorem 5.2.1]{P}. \par

There are invariant foliations $\mathcal W^{u}$ and $\mathcal W^{s}$ tangent to the bundles $E^{u}$ and $E^{s}$. The leaves of these foliations are the {\bf unstable leaves } and the {\bf stable leaves}. \par\vspace{.3em}

If $f$ is $C^{1+\alpha}$, the {\bf Pesin unstable manifold} of a point $x$ is the set
\begin{equation}
W^{+}(x)=\left\{y:\limsup_{n\to\infty} \frac{1}{n} \log d(f^{-n}(x),f^{-n}(y))<0\right\}.    
\end{equation}

$W^{+}(x)$ is an immersed manifold for a {\bf total measure set} $\cP$ ($\mu(\cP)=1$ for all $f$-invariant ergodic probability measures $\mu$) \cite{Pesin76}. In our setting,   $W^{u}(x)\subset W^{+}(x)$. If $\lambda^{c}_\mu(f)>0$, then $W^{u}(x)\subsetneq W^{+}(x)$ on $\cP$. This follows from the Pesin Stable Manifold Theorem: $\dim(W^{+}(x))$ equals the number of positive Lyapunov exponents of $f$, counted with their multiplicity \cite{Pesin76}.\vspace{.3em}

\begin{notation}\label{regular.points}
\begin{enumerate}
    \item $\cP$ is the total measure set of points $x$ for which $W^+(x)$ is an immersed manifold. 
    \item $\cW^{+}$ is {\bf the partition of} $\cP$ into the leaves $W^{+}(x)$.
\end{enumerate}    
\end{notation}
\vspace*{0.3em}
\begin{definition}\label{def.measurable.partition}
If $\mu$ is a probability measure, a {\bf $\mu$-measurable partition} $\eta$ is a partition satisfying - up to a set of $\mu$-measure zero - that the quotient space $\T^{3}\slash \eta$ is separated by a countable number of measurable sets (see \cite{Rokhlin1949}).
\end{definition}
\vspace{0.3em}
\begin{definition} For an ergodic invariant probability measure  $\mu$, and a partition $\mathcal F$ of a total measure set $\cP$ so that ${\mathcal F}(x)$ is an immersed manifold for every $x\in \cP$, a $\mu$-measurable partition $\eta$ is {\bf $\mu$-subordinate} to $\mathcal F$ if for all $x\in \cP$, 
\begin{enumerate}
 \item $\eta(x)\subset\mathcal F(x)$ has a uniformly bounded from above diameter inside $\mathcal F(x)$, 
 \item $\eta(x)$ contains an open neighborhood of $x$ inside $\mathcal F(x)$.
\end{enumerate}

 \end{definition}
\vspace*{2em}
From this point on, we will talk about ``subordinate partitions'' without explicitly mentioning any measure. All invariant probability measures will be considered ergodic, unless explicitly stated otherwise. \par\vspace{.3em}

$\eta$ is any $\mu$-measurable partition subordinate to $\mathcal{W}^+$ (see proof of its existence in \cite{Pesin76}). $\mu^{\eta}_{x}$ and $m_x^{\eta}$ are the $\mu$ and Lebesgue measures in $\eta(x)$ induced by the Riemannian structure of $\T^{3}$ on $W^+(x)$.
We will call them $\mu_{x}^{+}$ and $m_{x}^{+}$ when $\eta$ is clear from the context. 
\begin{definition}[SRB measure] If $f$ is $C^{1+\alpha}$, $\mu$ is an {\bf SRB measure} for $f$ if 
\begin{equation}\label{eq.def.srb}
\mu_x^{+}\quad \text{is equivalent to}\quad m_x^{+}\quad \mu-\text{a.e. }x,
\end{equation}
Equation \eqref{eq.def.srb} means that $\mu^{+}_{x}<< m^{+}_{x}\quad \text{and}\quad m^{+}_{x}<< \mu^{+}_{x}$.\: $\mu<<\nu$ means that $\mu$ is absolutely continuous with respect to $\nu$.
\end{definition}

With the above notations:
\begin{theorem}[Ledrappier-Young \cite{LY1}, Brown \cite{B}]\label{srb}  If $f$ is a $C^{1+\alpha}$ diffeomorphism , $\mu$ is an SRB measure of $f$ if and only if 
\begin{equation}\label{pesin.entropy.formula}
 h_{\mu}(f)  = \sum_{\lambda_i(f) > 0} \lambda_i(f).
\end{equation}
($h_{\mu}(f)$ is the sum of all positive Lyapunov exponents $\lambda_i$ counted with their multiplicity.)
\end{theorem}

Brown showed that this theorem holds in the $C^{1+\alpha}$ case \cite[Theorem 1.4]{B}.\vspace{.3em}

\begin{notation}
    The {\em support}
 of a probability measure $\mu$ is 
 $$\supp(\mu)=\left\{x:\, \mu(B_\eps(x))>0\quad\forall \,\eps>0\right\},$$
 where $B_\eps(x)$ is, as usual, the ball of radius $\eps>0$ centered at $x\in\T^3$.
 \end{notation}
 \begin{proposition}\label{supp.+.saturated} If $\mu$ is an SRB measure, then the set  $\supp(\mu)$ is {\em $W^+$ -saturated}:  $\mu$-almost every $x\in\supp(\mu)$, \[W^+(x)\subset \supp(\mu).\]
 \end{proposition}
 \begin{exercise}
     Find the reference for this. Better yet, try to prove it yourself.
 \end{exercise}
This is our use of Jacobians. 

\begin{definition}\label{jacobian} If $g:M\to N$ is a function, a measurable function $Jg:M\to \R$  is a {\bf Jacobian} of $g$ if 
$$ m_N(g(A))=\int_A Jg \, dm_M\qquad\forall\;A\quad\text{measurable}.$$
If $g$ is a diffeomorphism, then $Jg=\det (Dg)$ almost everywhere. $m_{M}$ and $m_{N}$ are the Lebesgue measures. 
\end{definition}
Pay attention to this notation
\begin{notation}
If $f:\T^{3}\to\T^{3}$ is partially hyperbolic, 
$$Jf^{+}=\left\{
\begin{array}{ll}
Jf^{u}&\text{ if }\lambda^{c}_{A}<0\\
Jf^{cu}&\text{ if }\lambda^{c}_{A}>0
\end{array}
\right..$$
$Jf^{u}$ is the determinant of $Df$ restricted to $E^{u}$, and $Jf^{cu}$ the determinant restricted to $E^{cu}.$
\end{notation}
\section{Preliminary rigidity results} \label{section.thm.A}
\begin{maintheorem}[SRB rigidity]\label{thmA}
If $f\in \cDA^{1+\alpha}(A)$ and $\mu$ is the measure of maximal entropy of $f,$ $\mu$ is SRB if and only if
\begin{equation}\label{condition.thm.A}
h_{top}(f)=h_{top}(A)=\lambda^{+}_{f}(p)\qquad\forall p\in\Per_{H}(f)\cap\supp(\mu)
\end{equation}
In this case, $\mu$ is the only SRB measure for $f$.
\end{maintheorem}
\begin{maintheorem}[Rigidity, ma non tanto]\label{thm.examples}
If $A$ has eigenvalues satisfying $\lambda^s_A < \lambda^c_A<0<\lambda^u_A$, there are $f\in\cDA(A)$ whose entropy maximizing measure $\mu$ is SRB and can have one of these properties:\vspace{.4em}

\begin{enumerate}
 \item $f$ is not Anosov. The measure of maximal entropy $\mu$ can be chosen to be Lebesgue or not. 
\item $f$ is Anosov and the Lyapunov exponents of $\mu$ are different from those of $A$. 
\end{enumerate}
\end{maintheorem}

The sufficiency of condition \eqref{condition.thm.A} in Theorem \ref{thmA} is not very complicated: the measure of maximal entropy $\mu$ is ergodic and hyperbolic, and $\sgn(\lambda^{c}(\mu))=\sgn(\lambda^{c}_{A})$ (Theorem \ref{thm.ures}). By A. Katok's closing lemma \cite{katok80} periodic points approximate the support of $\mu$,
 and there are periodic measures $\mu_{n}$ - supported on the orbits of these periodic points -  converging to $\mu$ in the weak-$*$ topology. Then:
\begin{eqnarray*}
 \lambda^{+}_{f}(\mu)&=&\int\log Jf^{+}(x)d\mu\\
 &=&\lim_{n}\int\log Jf^{+}(x)d\mu_{n}\\
 &=&\lim_{n}\lambda^{+}_{f}(p_{n})=h_{top}(f)=h_{top}(A),
\end{eqnarray*}
where $p_{n}$ belongs to $\supp \mu_{n}$ (see \eqref{lambda+}.) In our case, these periodic points can be chosen inside the support of $\mu$. See Claim \ref{pg.in.Kg}.\par

Since $\mu$ satisfies the Pesin formula, it is an SRB measure (Theorem \ref{srb}).
\begin{definition}\label{saturated.set} A set $L$ is {\bf saturated with respect to a foliation } $\F$ if for every $x\in L$, $\F(x)\subset L$. We also say that $L$ is {\bf $\F$-saturated}.
\end{definition}

\begin{definition}\label{minimal.set}
 A compact set $K$ is {\bf minimal with respect to a foliation } $\F$ if every non-empty compact $\F$-saturated set $L$ contained in $K$ satisfies $L=K$. We also say that $K$ is {\bf $\F$-minimal}. If $\T^3$ is minimal with respect to a foliation $\F$, then $\F$ is a {\bf minimal foliation}.
\end{definition}
\begin{theorem}\label{thm.ures}\cite{U}
 If $f\in\cDA^{1+\alpha}$, there exists a unique measure of maximal entropy, $\mu$. And
 
\begin{enumerate}
 \item  $\sgn(\lambda^{c}_{f}(\mu))=\sgn(\lambda^{c}_{A})$.
 \item If $\lambda^{c}_{A}<0$, there exists a unique $\cW^{u}$-minimal set $K_{u}=\supp\mu$.
\item If $\lambda^{c}_{A}>0$, there exists a unique $\cW^{s}$-minimal set $K_{s}=\supp\mu$.
\end{enumerate}
The topological entropy of $f$ satisfies $h_{top}(f)=h_{top}(A)$.
\end{theorem}
\subsection{Strategy to prove the necessity in Theorem \ref{thmA}}\label{subsection.strategy}
We will see that the cohomological equation \eqref{cohomological.equation} admits a continuous solution $\phi$ on $\supp\mu$. 
That is, there exists a continuous $\phi$ satisfying:
\begin{equation}\label{cohomological.equation}
\log Jf^{+}(x)-h_{top}(f)=\phi(x)-\phi(f(x)) \qquad{\forall x\in \supp \mu}. 
\end{equation}

For any periodic point $p$ in the support of $\mu$, \: $\mu_{p}$ is the ergodic invariant measure supported on the orbit of $p$. If there is a continuous function $\phi$ satisfying \eqref{cohomological.equation}, we can integrate \eqref{cohomological.equation} with respect to the measure $\mu_{p}$ and get $\lambda_{f}^{+}(p)=h_{top}(A)$. This proves Theorem \ref{thmA}. \par\vspace{.3em}

To obtain a continuous solution of the cohomological equation \eqref{cohomological.equation} we start with a ``formal solution'' of \eqref{cohomological.equation} and then prove that this formal solution is measurable and coincides with a continuous function $\phi$ $\mu$-almost everywhere in $\supp \mu$.\vspace{.3em}

\subsection{Proof of the necessity in Theorem \ref{thmA}}\label{sufficiency.thmA}
We introduce a formal solution of \eqref{cohomological.equation}. \par\vspace{.3em}

Since $f$ is a DA-diffeomorphism, there exists a semi conjugacy $h:\T^{3}\to\T^{3}$ such that 
\begin{equation}\label{semiconjugation}
h\circ f=A\circ h. 
\end{equation}
Assume that $Jh^{+}$ is formally defined. Then, from \eqref{semiconjugation}, we get 
$$J(h\circ f)^{+}=JA^{+}.Jh^{+}.$$ 
By the rules of the Jacobian,  $Jh^{+}(f(x)).Jf^{+}(x)=JA^{+}.Jh^{+}(x)$. Taking $\log$, 
 $$\log Jf^{+}(x)-h_{top}(A)=\log Jh^{+}(x)-\log Jh^{+}(f(x)).$$
So, $\log Jh^{+}$ is a formal solution of \eqref{cohomological.equation}. \par
We will see that $Jh^+$ is measurable and coincides with a continuous function $\mu$-almost everywhere. 
\subsection{The case $\lambda^{c}_{A}<0$}\label{subsection.lambda.neg}
Assume that the measure of maximal entropy $\mu$ is an SRB measure. Then $\lambda_\mu^c(f)<0$ (Theorem \ref{thm.ures}.) This implies that $W^{+}(x)=W^{u}(x)$ $\mu$-almost every $x\in\T^{3}$, since $\dim W^{+}(x)=$number of positive Lyapunov exponents=1.\par\vspace{.3em}

Given a point $z$, take a cubic box $C(h(z))$ with a side $L>0$ centered on $h(z)$ and subfoliated by the local ``linear foliations '' $E^s_A, E^{c}$, and $E^{u}_{A}$. For simplicity, assume that the axes $E^{\sigma}_{A}$ are orthogonal. Denote by $D(z)=h^{-1}(C(h(z)))$. $L$ can be taken to be uniform for all $z\in\T^{3}$. \par\vspace{.3em}

 $D(z)$ is a neighborhood of $z$ partitioned by bounded subarcs of $W^{+}(x)$, with $x\in D(z)$. For each $x\in D(z)$, $\mu^{+}_x$ is the conditional measure associated with the restriction of $\mu$ to $D(z)$, and $m^{+}_{x}$ is the measure induced by the Riemannian structure of $W^{+}(x)$ of the connected component of $W^{+}(x)$ intersected with $D(z)$ containing $x$.
 \begin{notation}\label{W.+.D}
$W_D^{+}(x)$ is the connected component of $W^{+}(x)$ intersected with $D(z)$ that contains $x$.   
 \end{notation}
\vspace{.5em}
Since $\mu^{+}_{x}<< m^{+}_{x}$ for $\mu$-almost every point $x$, there exists an $L^{1}$ density function $\rho_{x}^{+}$ such that for every interval $[a,b]^{+}\subset W^{+}_{D}(x)$, we have:

\begin{equation}\label{rho.u.m}
\mu^{+}_{x}([a,b]^{+})=\int_{a}^{b}\rho_{x}^{+}(\xi)dm^{+}_{x}(\xi) 
\end{equation}
\begin{remark}
   For a $C^{1+\alpha}$ diffeomorphism $f$, the density function $\rho^{+}_{x}$ is more than $L^{1}$: it is continuous and positive, and it satisfies the following Pesin-Sinai equations:
\vspace*{1em}
\begin{equation}\label{rho+}
\rho^+_{x}(\xi) =  \frac{\Delta^+(x,\xi)}{ \int_{W_D^{+}(x)} \Delta^+(x,\xi) dm^{+}_{x}},\qquad \Delta^+(x,\xi) = \frac{\prod_{i=1}^{+\infty} Jf^{+}(f^{-i}(x) )}{\prod_{i=1}^{+\infty} Jf^{+}(f^{-i}(\xi) )}. 
\end{equation}
\vspace*{1em}\\
Both $(x,\xi)\mapsto\rho^{+}_{x}(\xi)$ and $(x,\xi)\mapsto\Delta^{+}(x,\xi)$ are continuous. 
This is because $L>0$ is uniform, and then the sizes of $D(z)$ are continuous. (See \cite [Proposition 2, item (2)] {1982pesinsinai}.)
\end{remark}

\begin{claim}\label{claim.pushforward}
$$h_{*}(\mu^{+}_{x})=\frac{1}{L}m^{+}_{h(x)}\quad \text{for $\mu$-almost every $x$}.$$
\end{claim}
\bp
The claim is a consequence of the following lemma:\vspace{0.5em}
  \begin{lemma}\cite{U} \label{isomorphic} If $f\in\cDA^{1+\alpha}(A)$, then
$(f,\mu)$ and $(A,m)$ are metrically isomorphic, and
\begin{equation}\label{h.lleva.mu.en.m}
 h_{*}(\mu)=m,
\end{equation}
The set $h^{-1}(h(x))$ is connected and is contained in $W^{c}(x)$ for all $x\in \T^{3}$.
\end{lemma}
\vspace*{1em}
$h$ takes unstable leaves with respect to $f$ into unstable leaves with respect to $A$ (\cite{U}), and sends the measure $\mu^+(x)$ - equivalent to volume on $W^u(x)$ - into the normalized volume measure on $E_A^u(h(x))$. Then, it sends the volume on $W^u(x)$ into a measure equivalent to Lebesgue on $E_A^u(h(x))$. ({\bf Exercise:} Check that the claim is true from this point on.) \ep{Claim \ref{claim.pushforward}}
The claim \ref{claim.pushforward} implies that $h$ has an unstable Jacobian $Jh^+$. $Jh^+$ only depends on the Riemannian metric. So, it does not depend on the choice of the box $C(h(z))$ (see Definition \ref{jacobian}). This implies the following \vspace{.3em}

\begin{claim}
For all $x\in\T^3$ and $\mu$-almost every $x \in D(z)$, all $f$-unstable segments $[a,b]^{+}$ in $W^{+}_{D}(x)$ satisfy this equation:
\begin{equation}\label{rho.u.m.h}
 \mu^{+}_{x}([a,b]^{+})= \frac{1}{L}m^{+}_{h(x)}([h(a),h(b)]^{+})=\frac{1}{L}\int_{a}^{b} Jh^{+}(\xi)dm^{+}_{x}(\xi). 
\end{equation}
\end{claim}

The formulas \eqref{rho.u.m} and \eqref{rho.u.m.h} together imply that
$$Jh^{+}(\xi)= L \rho_{x}^{+}(\xi)\qquad \mu-\text{almost every }\xi \in   D.$$

 Since $\log \rho^{+}_{x}$ is continuous, $\log Jh^{+}$ extends to a continuous function on $D(z)\cap \supp \mu$. So, $Jh^+$ extends to a continuous function on $\supp \mu$. 
\ep{Theorem \ref{thmA}, case $\lambda^c_A<0$}

\subsection{The case $\lambda^{c}_{A}>0$.} 
\begin{proposition}
    \label{infiniterad}
    If $f\in\cDA^{1+\alpha}(A)$ with $\lambda^c_A>0$ and the measure of maximal entropy $\mu$ is an SRB probability measure, then $$W^+(x)=W^{cu}(x)\qquad\text{for}\quad\mu-\text{almost every }x.$$
\end{proposition}
The above proposition is a particular case of \cite[Proposition 5.2]{RHUY}. We include a proof of it in our simpler setting, to improve readability.\par\vspace{1em}

\bp
$\lambda^c_f(\mu)>0$ (Theorem \ref{thm.ures}), so $\dim W^+(x)=2$, and $W^+(x)$ contains a $2$-dimensional disc centered at $x$ for $\mu$-almost every $x$.\par\vspace{.3em}

The easy part is to prove that $W^+(x)$ is contained in $W^{cu}(x)$ for $\mu$-almost every $x$, and it is left to the reader. \par\vspace{.3em}

The difficult part is the other inclusion. $W^+(x)$ contains a 2-disc centered at $x$. If this disc does not "colapse" under the action of $h$, and $h(W^+(x))$ contains a disc centered at $h(x)$, then the linear action of $A$ will make this disc grow uniformly - playing with successive $A^n\circ h\circ A^{-n}$ -. The fact that $h$ is {\bf proper} (the reader can look for this meaning) will show that $W^+(x)$ occupies the whole immersed plane $W^{cu}(x)$.\par\vspace{.3em}

$h$ is injective $\mu$-almost everywhere (Lemma \ref{isomorphic}). Then, we can consider that $h$ is injective in every $x\in\cP$ (Notation \ref{regular.points}). This implies that the $h$-image of $W^+(x)\cap W^c(y)$ contains a neighborhood of $h(y)$ - a positive radius center segment centered at $h(y)$ - in $E^c_A(h(y))$ for all $x\in\cP$ and $y\in W^+(x)$ (recall that $h$ only identifies points in the same center leaf, see Lemma \ref{isomorphic}.) \par\vspace{.3em}

%
%
%\begin{exercise}  \begin{enumerate}
 %       \item $h(\cP)$ has full Lebesgue measure $m$.
%        \item For some $\delta>0$ there exists a positive $m^u_x$-measure set of points $y\in h(W^u(x))$,\; for which  $h(W^c(y))$ contains a segment of radius $\delta$, centered at $h(y)$.
%        \item Conclude that $h(W^+(x))\subset E^u_A(h(x))$ contains an $m$-positive measure set, $B$. {\em Hint:} there is some Fubini magic trick involved.
%        \item Use Poincar\'e Recurrence Theorem, and the $A$-ergodicity of $m$ to show that all points $w\in h(\cP)$ visit $B$ infinitely many times, both in the future and the past. (We are only interested in the past now, though.)
%    \end{enumerate}
%\end{exercise}
\vspace{0.3em}
There is a full measure set $\cR$ on which $h$ is injective. $h$ only collapses center segments, so $h(W^+(x)\cap W^c(x))$ contains a neighborhood in $E^c_A(x)$ for all $x\in\cR$. Because $m(h(\cR))=1$, there is an $m$-positive measure set $B$ so that $h(W^+(x)\cap W^c(c))$ contains a segment of radius $\delta>0$ centered at $x$ for all $x\in B$.\par
The ergodicity and invariance of $m$ with respect to $A$ imply that $h(W^+(x)\cap W^c(x))=E^c_A(h(x))$. Then $W^+(x)\cap W^c(x)=W^c(x)$ for $\mu$-almost every $x\in\cR$.\par
The reader may complete the details. 
\ep{Proposition \ref{infiniterad}}     
\vspace{0.3em}
\begin{proposition}\label{proposition.mu.x0}
    If $z$ is any point in $\T^3$, and $C=C(h(z))$ is a cubic box with a side $L>0$ (as defined in Subsection \ref{subsection.lambda.neg}), call $D=D(z)=h^{-1}(C(h(z))$. Then, for all $x_0\in D\cap \cP$ and all measurable sets $B\subset W^+_D(x_0)=W^{cu}_D(x_0)$,    \begin{equation}\label{equation.m`u+1}
\mu^{+}_{x_{0}}(B)=\frac{1}{L^{2}}\int_{B} Jh^{+}(\xi)dm^{+}_{x_{0}}(\xi).
\end{equation}
\end{proposition}
\begin{proof}
$D$ is foliated by discs $W^{cu}_D(x)$, and for all measurable sets $B$ contained in $W^+_D(x_0)=W^{cu}_D(x_0)$:
\[\mu^{+}_{x_{0}}(B)=\int_{B}\rho_{x_{0}}^{+}(\xi)dm^{+}_{x_{0}}(\xi),\]
(the function $\rho_{x_{0}}^{+}$ is defined by Equation \eqref{rho+}. 
)

Since $h_{*}(\mu^{+}_{x})=\frac{1}{L^{2}}m^{+}_{h(x)}$  $\mu$-almost every $x\in D$ (the proof is as in Subsection \ref{subsection.lambda.neg}),  
\[\mu^{+}_{x_{0}}(B)= \frac{1}{L^{2}} m^{+}_{h(x_{0})}(h(B)).\]
Then, 
\[\mu^{+}_{x_{0}}(B)= \frac{1}{L^{2}}\int_{B} Jh^{+}(\xi)dm^{+}_{x_{0}}(\xi).\]
\end{proof}
In Subsection \ref{subsection.lambda.neg} ($\lambda^{c}_{A}<0$) the continuity of $\rho_{x}^{+}$ follows directly from formulas \eqref{rho+} and from uniform hyperbolicity. \par\vspace{.3em}

This case is more elaborated. The continuity of $\rho^{+}_{x}$ is only clear when restricted to a local Pesin unstable manifold $W^{+}_D(x)$, with $x\in\cP$. Its continuity on $\T^{3}$ is not trivial.\par\vspace{.3em}

 We build a continuous extension of $ \rho^{+}_{x_{0}}$ that coincides with $\mu$-almost everywhere with $\frac{1}{L^2} Jh^{+}$. 
We then extend $Jh^{+}$ continuously on $D$.\par\vspace{.3em}

\begin{definition}
    If $x$ and $y$ belong to some $D(z)$, the {\bf stable holonomy map} between $x$ and $y$ is the continuous function
    \begin{eqnarray*}
        H:&W^{+}_D(x)\to &W^+_D(y)\\
        &w\hspace{1.3em}\longmapsto& W^s_D(w)\cap W^+(y).
    \end{eqnarray*}
    
\end{definition}
If $x\in D$ and $H$ is the stable holonomy map between $x$ and $x_{0}$, define
\vspace{.5em}
\begin{equation}\label{rho+holonomy}
\rho^{+}_x:=\rho^{+}_{x_{0}}(H(x))JH^{-1}(x), 
\end{equation}

\vspace{0.5em}
 where $JH:D\to \R^+$ is the Jacobian of $H$. The Jacobian $JH$ is continuous and positive \cite[Theorem 2.1]{1972pughshub}, and so is $\rho^{+}$. $JH$ is continuous on $D$. The formula for $JH$ on $D$ is
 \begin{equation}\label{holonomy}
 JH(x)\:=\: \prod_{j=0}^\infty \frac{Jf^{cu}(f^j(x))}{Jf^{cu}(f^j(H(x)))}
 \end{equation}
\vspace{0.3em}

The uniform convergence of the formula in $D$ is proved in \cite[Theorem 2.1]{1972pughshub}  \eqref{holonomy}. (The continuity of $JH$ on a given transversal appears more frequently in the literature, but the continuity of $JH$ on $D$ follows from formula \eqref{holonomy}. See below.)\par\vspace{.3em}

The continuity of $JH$ on $D$ follows from the H\"older continuity of $Jf^{cu}$: \vspace{.3em}

If $y\in W^s_D(x)$, then 
 $$ \left|\log JH(x)- \log JH(y)\right| \quad\leq\quad \left|\sum_{j=0}^\infty  \log Jf^{cu}(f^j(x)) - \log Jf^{cu}(f^j(y))  \right|\,,$$
because $H(x)=H(y)$. Since  $\log Jf^{cu}$ is H\"older, there are positive constants $K,C$, and $\alpha$ such that
 $$ |\log JH(x)- \log JH(y)|\quad \leq\quad \sum_{j=0}^\infty C\, d(f^j(x), f^j(y)^\alpha\quad\leq\quad K\,d(x,y)^\alpha\,.$$
 
The last inequality is because $y\in W^s_D(x)$, so $JH$ is H\"older continuous on stable plaques. \par\vspace{.3em}
\begin{claim}
The continuous function $\rho^{+}$ coincides with $\rho^{+}_{x}$ (Equation \eqref{rho+holonomy}) $\mu$-almost every $x\in D$.   
\end{claim}
\begin{notation}\label{W.s.u.C}
    For a fixed $z'\in h^{-1}(z)\in\T^3$, $C=C(z')$ is $h(D(z))$ and $W^s_C(x')$ is the connected segment of $W^s(x')\cap C$  that contains $x'$. $W^u_C(x')$ is the connected $2$-disc (also called plaque) $W^u_A(x')\cap C$ that contains $x'$.\par
    We adopt the convention that, fixed $D=D(z)$, $x',y'$ are the $h$-preimages of $x$,$y$ that belong to $C=C(z')$.
\end{notation}
\begin{proof}
    $H_{A}$ - the stable holonomy map for $A$ - is isometric on unstable discs with respect to the measure $m^+$ for all $x,y\in C$ such that $y\in W^s_C(x)$. This means that:
\begin{equation}
m^+_y(H_A(B))=m^+_x(B)
\end{equation}
for all $m^+_x$-measurable sets $B$ contained in $W^u_C(x)$. ($H_A$ is a homeomorphism, so $H_A(B)$ is $m^+$-measurable.) {\bf Exercise:} check it. 
\vspace{.3em}
\begin{notation}
    If $H:E_x\to E_y$ is a continuous map, and $\mu_x$ is a measure on $E_x$, the {\bf push-forward measure} of $\mu_x$ on $E_y$ is the following measure:
    \begin{equation}
        H_*\mu_x(B)=\mu_x(H^{-1}(B)),\qquad\forall \;\text{measurable sets }B\subset E_y.
    \end{equation}
    $E_x$ and $E_y$ are any two random sets. 
\end{notation}
\vspace{.3em}
Choose $x'\in h^{-1}(x)$ and $y'\in h^{-1}(y)$, so that $x',y'\in D$. 
({\bf Check:} $h_{*}(\mu^{+}_{x'})= \frac{1}{L^{2}} m^{+}_{x}$).\par
Set $\nu_{x'}:=\left(H\right)_*\mu^+_{x'}$ and 
\begin{equation}\label{formula.nu}
\nu(B_D)=\int_{\T^3}\nu_{x'}(B_D\cap W^+_D(x')d\mu(x')\qquad\forall \,B_D\subset D\quad\text{measurable set}
\end{equation}
{\bf Check:}
\begin{enumerate}
    \item $H_*(\nu_{x'})=\frac{1}{L^2}m_x$.
    \item The density of $\nu$ is $\rho^+$ (Equation \eqref{rho+holonomy}).
\end{enumerate}
\end{proof}
\vspace{.5em}
A computation as in the proof of Proposition \ref{proposition.mu.x0} implies that $\frac{1}{L^{2}} Jh^{+}$ coincides $\mu$-almost everywhere with $\rho^{+}$ in $D$.  This implies that $Jh^{+}$ extends to a continuous function on $\T^3$. To avoid an excess of notation, we call this continuous extension $Jh^{+}$. $\phi=\log Jh^{+}$ is a continuous function that satisfies the cohomological equation \eqref{cohomological.equation}. This implies that $h_{top}(f)=\lambda^{+}_{f}(p)$ for all periodic points $p$, and finishes the proof of Theorem \ref{thmA}.
\ep{Theorem A}
\begin{cor}\label{corolario.fragilidad.propiedad}
If $\mu_{mme}(f)$ is the (unique) measure of maximal entropy for $f\in \cDA$, the set
$${\mathcal{F}\hspace{-.3em}r}=\{f\in \cDA: \mu_{mme}(f)\quad\text{is not an SRB measure}\}$$
is $C^1$-open in $\cDA$. $\mathcal{F}\hspace{-.3em}r\cap \diff^{\infty}(M)\cap\cDA$ is $C^\infty$ dense in $\cD$. 
\end{cor}

\bp If $f\in \cDA^{1+\alpha}$, then the unique measure of maximal entropy $\mu_{mme}(f)$ is hyperbolic (Theorem \ref{thm.ures}). Then hyperbolic periodic points are dense in the support of $\mu_{mme}(f)$, by A. Katok's closing lemma \cite{katok80}. By Theorem \ref{thmA}, $f\notin\Fr$ if and only if 
$$h_{top}(f)=h_{top}(A)=\lambda^{+}(p)\qquad\forall p\in\Per(f)\cap \supp(\mu_{mme}(f))\qquad \eqref{condition.thm.A}$$
\begin{lemma}\label{lemma.equality.Fr}
The set $\cDA\setminus\left(\Fr\right)
$ is equal to the set $$\left\{g\in\cDA:\:\exists\, p\in\Per_H(g)\cap\supp(\mu_{mme}(g))\;\text{so that } \lambda^+_g(p)\ne h_{top}(A)\right\}, $$
where $\Per_H(g)$ stands for the set of hyperbolic $g$-periodic points. $\cDA\setminus \Fr$ is $C^1$-open.
\end{lemma}
\bp
    {\bf Exercise:} Prove the equality. \par
    For each $f\in\cDA$, $g\mapsto \lambda^+_g(p_g)$ is a continuous function whose domain is a small $C^1$-neighborhood of $f$ ($p_g$ is the continuation of any given $p_f\in\Per_H(f)$.)\par
    So, if either $\lambda^+_g(p_g)>h_{top}(A)$ or $\lambda^+_g(p_g)<h_{top}(A)$ for some $g\in\diff^1(\T^3)$, this still holds in a $C^1$-neighborhood of $g\in\diff^1(\T^3)$.
\ep{Lemma \ref{lemma.equality.Fr}}
\begin{lemma}
 It is possible to $C^\infty$-perturb $f$ so that condition \eqref{condition.thm.A} is not satisfied for the continuation of an $f$-periodic point $p$.     
\end{lemma}
\bp\newline
{\bf Case 1 -} $\mathbf{\supp\mu_{mme}=\T^3:}$\quad
Compose $f$ with the time-one map of a  $C^\infty$ vector field $X$ supported in a small neighborhood of $p$. 
Any divergence-free $X$ is $m$-invariant. The flow is also $m$-invariant.  A computation like this can be found in Section \ref{proof.thm.B}.\vspace{.3em}
\ep{Case 1}
\vspace{.5em}
{\bf Case 2 -  }$\mathbf{\supp\mu_{mme}:}$ After a $C^\infty$-perturbation of $f$, it could happen that the continuation $p_g$ of an $f$-periodic point $p_f$ no longer belongs to $\supp\mu_{mme}(g)$. In fact, this is never the case in our setting: 

 If $K_{f}=\supp(\mu_{mme}(f))$ and $\lambda^{c}_{A}>0$, then $K_{f}=\T^{3}$ by Lemma \ref{lemma.s.minimal}. If $\lambda^{c}_{A}<0$, $K_{f}$ is the unique minimal set saturated by $W^{u}$-leaves (Theorem \ref{thm.ures}). Since the Hausdorff limit of compact $u$-saturated sets is compact and $u$-saturated, it follows that $f\mapsto K_{f}$ is lower semicontinuous in the $C^{r}$-topology, for all $r\geq 1$.  Then:
\begin{claim}\label{pg.in.Kg}
    If $p_f\in K_{f}$ is an $f$-hyperbolic periodic point with $\lambda_f^c(p_f)<0$, 
there exists a $C^{1}$-neighborhood ${\mathcal U}(f)$ of $f$, so that the continuation $p_{g}$ of $p_f$ for $g$ belongs to $K_{g}=\supp(\mu_{g})$ for all $g\in{\mathcal U}(f)$ .  
\end{claim} 
\bp
$\dim W^s(p_f)=2$ and $W^s(p_g)$ -also two dimensional - is $C^1$-close to $W^s(p_f)$ if $g$ is $C^1$-close to $f$. The lower semicontinuity of $f\mapsto K_{f}$ implies $K_g\cap W^s(p_g)\ni z$. If $N$ is the period of $p_g$, then $W^s(p_g)\cap K_g\ni g^{kN}(z)\to p_g$ when $k\to\infty$ ($K_g$ is $g$-invariant and $W^s(p_g)$ is $g^N$-invariant).
    $p_{g}$ belongs to $K_{g}$, because $K_g$ is closed.
\ep{Claim \ref{pg.in.Kg}+Case 2}

The Corollary follows from Theorem \ref{thmA}.
\ep{Corollary \ref{corolario.fragilidad.propiedad}}
\begin{lemma}\label{lemma.s.minimal}
 If the center Lyapunov exponent  $\lambda^c(A)$ of $A$ is positive, and the measure of maximal entropy $\mu_{mme}(f)$ of $f\in\cDA^{1+\alpha}(A)$ is an SRB measure, then $\supp(\mu_{mme}(f))=\T^3$.  The stable foliation $\mathcal W^{s}$ is a minimal foliation (Definition \ref{minimal.set}). 
\end{lemma}
\begin{proof}
$\supp(\mu_{mme}(f))$ is the only minimal set with respect to the stable foliation $W^s$ of $f$ (Theorem \ref{thm.ures}).\par 
$\supp(\mu_{mme}(f))$ also contains $W^+(x)$ for $\mu_{mme}(f)$-almost every $x\in\supp_{mme}(f)$ - see Proposition \ref{supp.+.saturated}.\par
So, take $W^+_{loc}(x_0)$ contained in $\supp(\mu_{mme}(f))$. It is transverse to the foliation $W^s$ in a neighborhood of $x_0$. $W^s(x)\subset \supp(\mu_{mme}(f))$ for all $x\in W^+_{loc}(x_0)$ ({\bf Why?}). \par
This implies that $\supp(\mu_{mme}(f))$ contains a non-empty open set $O$,
$$O=\bigcup\left\{W^s_{loc}(x):\: x\in W^+_{loc}(x_0)\right\}.$$
Since $\supp(\mu_{mme}(f))$ is the only minimal set with respect to the stable foliation of $f$, 
$$\supp(\mu_{mme}(f))\subset \overline{W^{s}(x)}\qquad \forall x\in \T^{3}. $$

\vspace{.5em}  
Then $O\cap W^s(x)\ne\emptyset$ for every $x\in\T^3$ ({\bf Check}).
Then $\supp(\mu_{mme}(f))=\T^3$ and $W^s$ is a minimal foliation.
\end{proof}
%
% Hence, we can make a $C^{\infty}$-perturbation of $f$ such that $\lambda^{u}_{g}(p)\ne\lambda^{u}_{A}$ for some $p\in\Per(g)\cap \supp(\mu_{g})$. This condition is $C^{1}$-open. Since now it is straightforward to perturb any periodic point to change its Lyapunov exponents, we can do this $C^{1}$-near the Anosov linear map so that $f$ is still Anosov, $\mu$ its mme is an SRB. Still, the Lyapunov exponents are different from those of $A$ (as a vector).
\section{Proof of Theorem \ref{thm.examples}}\label{proof.thm.B}

%\subsection{Case $\lambda^{c}_{A}<0$}
Let $A:\T^{3}\to \T^{3}$ be a linear Anosov diffeomorphisms with $\lambda^{c}_{A}<0$.
We will perform a small perturbation in a neighborhood of the fixed point along the center-stable direction, \`a la Ma\~n\'e.\par
As in \cite{RHUY}, two things have to be taken care of: 
\begin{enumerate}
    \item the perturbation has to be partially hyperbolic,
    \item the modulus of the expansion along the one-dimensional unstable direction has to be the same.
\end{enumerate}
 The unstable direction of the perturbation $f$ will be different from the unstable direction of $A$.
 \vspace{.8em}

%First not that can make a $C^{\infty}$-perturbation of $f$ such that $\lambda^{u}_{g}(p)\ne\lambda^{u}_{A}$ for some $p\in\Per(g)\cap \supp(\mu_{g})$. This condition is $C^{1}$-open. Since now it is straightforward to perturb any periodic point to change its Lyapunov exponents, we can do this $C^{1}$-near the Anosov linear map so that $f$ is still Anosov, $\mu$ its mme is an SRB. Still, the Lyapunov exponents are different from those of $A$ (as a vector).

%\newline\par

%%%%%%%%%%%%%%%%%%%%%%%%%%%%%%%%%%%%%%%%%%%%%%%%%%%%%%%%%
%%%%%%%%%%%%%%%%%%%%%%%%%%%%%%%%%%%%%%%%%%%%%%%%%%%%%%%%%
 Consider a vector field $X:\T^{3}\to T\T^{3}$ supported in a small ball containing a fixed point $p$ of $A$. $X$ satisfies:
\begin{enumerate}
 \item $X(p)=(0,0,0)$,
 \item $X(x)\in E^{cs}_{A}(x)$ for all $x\in\T^{3}$. \label{cs.invariant}
 \item\label{divergence} If we want the perturbation to be non-conservative, it is enough to ask that the divergence of $X$ at $p$ satisfy $\div(X)(p)\ne 0$,
\end{enumerate}
There is a lot of freedom to fulfill these requirements. \par
\vspace{.5em}

Let $\varphi_{t}:\T^{3}\to\T^{3}$ be the flow of $X$. $f:\T^{3}\to \T^{3}$ is
\begin{equation}\label{examples}
 f(x)=A\varphi_{1}(x)
\end{equation}
The choice of the time-one map of $\left\{\varphi_t\right\}_{t>0}$ is nothing special. One can choose any $t>0$. The condition (\ref{cs.invariant}) above implies that $f(E^{cs}_{A}(x))=E^{cs}_{A}(f(x))$, then the foliation $E^{cs}_{A}$ is $f$-invariant. ({\bf Check.}) \par
This implies:
\begin{claim}\label{lyapunov.perturbado}
The unstable Lyapunov exponent of $f$ is defined for all $x\in \T^{3}$ and  $$\lambda^{u}_{f}(x)=\lambda^{u}_{A}.$$
\end{claim}
\bp The proof of this is classic. You can find it, for example, in \cite{SW}. The original proof is due to Ma\~n\'e \cite{1995Manhe}. A brief explanation is that the unstable bundle of $f$ is contained in a uniformly sized cone around the unstable direction of $A$. The fact that $f$ takes $E^{cs}_{A}$-planes to $E^{cs}_{A}$-planes implies that the length of $f^{n}(\alpha)$ is proportional to $[\lambda^{u}_{A}]^{n}*$(length $\alpha$) for each $f$-unstable segment $\alpha$. This implies the claim. \ep{Claim \ref{lyapunov.perturbado}}

\par
\begin{claim}\label{unique.gibbs}
$f$ has a unique $u$-Gibbs measure $\mu$. $\mu$ is and SRB measure and the measure of maximal entropy.
\end{claim}
\bp 
There is always a Gibbs measure for a partially hyperbolic diffeomorphism $f$ \cite{1982pesinsinai}. Dolgopyat has written an excellent survey on the topic \cite{D}. \par

From \cite{LY}, 
$$h_{\mu}(f)\geq \lambda^{u}_{f}(\mu).$$
Claim \ref{lyapunov.perturbado} implies
$$\lambda^{u}_{f}(\mu)=\lambda^{u}_{A}=h_{top}(A).$$
Theorem \ref{thm.ures} implies $h_{top}(A)=h_{top}(f)$. 
Then $\mu$ is the only measure of maximal entropy. \par $\lambda^{c}(\mu)<0$ (Theorem \ref{thm.ures}), then the unstable Pesin manifold $W^+(x)$ of $x$ coincides with the strong unstable manifold $W^{u}(x)$ of $x$, $\mu$-almost every $x\in\T^3$. \par
So, $\mu$ is also an SRB measure. \ep{Claim \ref{unique.gibbs}}
\vspace{.5em}

 If $X$ is $C^{1}$ close to $0$, $f$ in \eqref{examples} is Anosov. 
 \par
 
We can also choose a vector field $C^0$, but not $C^1$ close to $\vec{0}$ $X$ that keeps the property (\ref{cs.invariant}) true.  Then $f$ as in \eqref{examples} is a DA diffeomorphism, $f\in\cDA(A)$. There is plenty of freedom to find a non-Anosov $f$ that satisfies condition (1) of Theorem \ref{thm.examples}. ({\bf Check}.) \par
\vspace{.5em}
With the same technique, it is also possible to find $f$ so that
$$\lambda^c(f)\ne\lambda^c_A\qquad\text{and}\qquad \lambda^s(f)\ne\lambda^s_A$$
satisfies condition (2). ({\bf Check}.)\ep{Theorem \ref{thm.examples}}

\begin{question}\label{question}
Are there examples that fulfill the conditions of Theorem \ref{thm.examples} that are neither Anosov nor volume-preserving when $\lambda^{c}_{A}>0$?
\end{question}

 \section{Proof of the Main Theorem}\label{section.proof.mainthm}
 
The theorem below assumes that $f$ is  a $C^{\infty}$ volume-preserving {\bf Anosov} map satisfying \eqref{lyapunov.rigidity}, plus that all eigenvalues are real and distinct, and concludes that the conjugacy is smooth. \par 
Many rigidity results assume that $f$ is Anosov, for instance \cite{MM1,MM2,Lla,LlaM,GG,GKS1,GKS2,SY} just to mention a few. \par
\ref{corolario.saghin.yang}

\begin{theorem}\cite{SY}
    If $f\in\cDA^\infty(A)$ is a volume-preserving Anosov diffeomorphism  with
simple real eigenvalues with distinct absolute values and
$$\lambda^s(f)=\lambda^s_A,\qquad \lambda^c(f)=\lambda^c_A,\qquad\text{and}\qquad \lambda^u(f)=\lambda^u(A),$$
 then f is $C^\infty$ conjugate to $A$.
\end{theorem}

Suppose that $f$ is in $\cDA^{1+\alpha}(A)$ and 
 $$\lambda^s(f)=\lambda^s_{A},\qquad \lambda^c(f)=\lambda^c_A,\qquad\text{and }\qquad\lambda^u(f)=\lambda^u_A.$$
We want to prove that $f$ is Anosov. Assume that $\lambda^{c}(f)=\lambda^{c}_{A}>0$.\par\vspace{.5em}
From the Pesin entropy formula \eqref{pesin.entropy.formula} and the equalities above it follows that $m$ is the entropy maximizing measure for $f$, since
$$h_{m}(f)=\lambda^{+}(f)=\lambda^{+}_{A}=h_{m}(A)=h_{top}(A).$$
The semiconjugacy $h$ takes the entropy maximizing measure for $f$ into Lebesgue measure by Theorem \ref{thm.ures}. So,
$$h_{*}(m)=m.$$

%%%%%%%
\subsection{Basic notations}
These concepts can be found in \cite{LY1}. We briefly summarize them here for self-containment.
\begin{enumerate}
\item A $\mu$-measurable partition is {\bf increasing} with respect to a $\mu$-preserving homeomorphism $g:\T^{3}\to \T^{3}$ if $\eta$ {\bf refines} $g\eta$, i.e. $g(\eta(x))\supset\eta(g(x))$ $\mu$-almost every $x$, where $\eta(x)$ denotes the element of $\eta$ that contains $x$. We denote a $g$-increasing partition $\eta$ by $\eta>g\eta$.
\item For $\mu$-measurable partitions $\eta$ and $\nu$, $H_{\mu}(\eta\mid\nu)$ denotes the {\bf mean conditional entropy} of $\eta$ given $\nu$, i.e.

\[
H_\mu(\eta \mid \nu) =-\int_{T^{3}} \mu_{\eta}^{x}(\nu(x))\,.\,\mu_{\eta}^{x}(\eta(x) \mid \nu(x))\,.\,\log \mu_{\eta}^{x}(\eta(x)\mid \nu(x))\:\:d\mu(x), 
\]  
where 
\[ \mu_{\eta}^{x}(\eta(x)\mid\nu(x)) = \frac{\mu_{\eta}^{x}(\eta(x) \cap \nu(x))}{\mu_{\eta}^{x}(\nu(x))}, \]
and if \( \mu_{\eta}^{x}(\nu(x)) = 0 \) then we omit it, adopting the usual convention $0.\log 0=0$.  
$\mu_{\eta}^{x}$ is the normalized probabiliy meausre $\mu$, on the set $\eta(x)$.
\end{enumerate}
\vspace{.5em}

\subsection{Strategy} 
\begin{enumerate}
 \item Obtain an increasing partition $\eta$ subordinate to $E^{c}_{A}$ for which $H_{m}(\eta|A\eta)=\lambda^{c}_{A}$.
 \item Pull back the partition, defining $\xi=h^{-1}(\eta)$.
  \item Show that $H_{m}(\xi|f\xi)=\lambda^{c}(f)=\lambda^{c}_{A}$.
\end{enumerate}
\vspace*{1em}

For any partition $\eta$ subordinate to $E_{A}^{c}$ (the $A$-invariant center foliation) define $\xi(\eta)$ as the partition subordinate to $\F^{c}$, the $f$-invariant center foliation obtained by taking the pre-image  of $\eta$ via the semiconjugacy $h$. 
The partition $\xi(\eta)$ is well defined $m$-a.e. $x$ and is subordinate to $\F^{c}$ since $W^{c}_{f}(x)\subset W^{+}_{f}(x)$ $m$-almost every $x\in \T^{3}$ by the Lemma \ref{infiniterad}.\par\vspace{.3em}

 $m_{x}^{f,c}$ is the Lebesgue measure induced in $\xi(x)$ by the Riemannian structure of $\T^{3}$ on $W^{c}_{f,loc}(x)$. For $m$-almost every $x\in \T^{3}$, $\mu^{f,c}_{x}$ is the measure given by the Rokhlin decomposition, the one satisfying for any measurable set $B\subset \T^{3}$ the property:
\begin{equation}\label{rokhlin}
 m(B)=\int_{\T^{3}}\mu^{f,c}_{x}( \xi(x)\cap B)dm.
\end{equation}
\begin{lemma} \label{lemma.equivalent.measures} There is a partition $\xi$ subordinate to $\F^{c}_{f}$ so that the measures $m^{f,c}_{x}$ and $\mu^{f,c}_{x}$ are equivalent  for $m$-almost every $x\in\T^{3}$.  
\end{lemma}
\bp We loosely follow \cite[Th\'eor\`em 4.8]{ledrappier1984}. See \cite[Theorem A]{LY1} for a friendlier version. Pesin's formula implies absolute continuity. We use this idea adapted to the center foliation. \par
$m$ is ergodic. $\xi$ is an increasing partition subordinate to $\F^{c}$, see below. We show that $H_{m}(\xi|f\xi)=\lambda^{c}(f)=\lambda^{c}_{A}$ implies that $\mu^{f,c}_{x}$ is equivalent to $m^{f,c}_{x}$ for $m$-a.e. \vspace{.3em}

\begin{claim} \label{partition.subordinate.A}There is a partition $\eta$ subordinate to $E^{c}_{A}$ so that
  $H_{m}(\eta|A\eta)=\lambda^{c}_{A}$.
\end{claim}
\bp
Any partition $\eta$ subordinate to $E^{c}_{A}$ satisfies 
\begin{equation}\label{entropia.condicional.particion}
H_{m}(\eta|A\eta)=-\int\log \mu^{A,c}_{y}(\eta(y))dm(y).
\end{equation}
$\mu^{A,c}_{y}$ are the conditional measures associated with the partition $A\eta$. \par
If ${\mathcal N}$ is a Markov partition associated with $A$, one can obtain a partition subordinated to $E^{c}_{A}$ by taking intersections with the center lines of $E^{c}_{A}$ and considering connected components. This new partition $\eta$ satisfies the claim. \par\vspace{.3em}

$\eta$ is increasing with respect to $A$\, ($\eta(Ay)\subset A\eta(y)$ for $m$-almost all $y\in\T^{3}$). The reader can check it herself. Then, the partition $\eta$ is {\bf generating}:
$$\bigcap_{n\in\N} A^{-n}\eta(A^{n}y)=\{y\}.$$
Then,
$$\eta(z)\subset A\eta(y)\quad\forall z\in A\eta(y),\quad m{\text -a.e. }y\in \T^{3}.$$
The conditional measures for both $\eta$ and $A\eta$ are the normalized Lebesgue measure along the center space $E^{c}_{A}$ on each of the elements of each partition. So,
$$\mu^{A,c}_{y}(\eta(y))=\frac{|\eta(y)|}{|A\eta(y)|}=\frac{|A^{-1}\eta(y)|}{|\eta(A^{-1}(y))|}=\frac{\exp(-\lambda^{c}_{A})|\eta(y)|}{|\eta(A^{-1}(y))|},$$ where $|.|$ denotes the length.\\
Applying these equalities to Formula \eqref{entropia.condicional.particion}, we get:
$$H_{m}(\eta|A\eta)=-\int(\log(\exp(-\lambda^{c}_{A}))+\log|\eta(y)|-\log|\eta(A^{-1}(y))|)dm(y).$$
The fact that $m$ is $A$-invariant implies $H_{m}(\eta|A\eta)=\lambda^{c}_{A}$.
\ep{Claim \ref{partition.subordinate.A}}
Take the partition $\xi=h^{-1}(\eta)$. Check that the elements $\xi(x)$ of $\xi$ are connected. The claim below is a consequence of this definition.
\begin{claim}\label{claim.increasing.partition}
 The partition $\xi$ is increasing.
\end{claim}
\begin{claim}\label{metric.entropy.f}
$H_{m}(\xi|f\xi)=\lambda^{c}(f)$. 
\end{claim}
\bp
$h$ is a metric isomorphism between $(A, m)$ and $(f, m)$. 
Then
\begin{equation}\label{eq.rokhlin} h^{-1}_{*}(\mu^{A,c}_{x})=\mu^{f,c}_{h^{-1}(x)}\qquad m\text{-almost every }x\in\T^{3}.
\end{equation}

So, 
$$m(B)=m(h(B))=\int  \mu^{A,c}_{x}(h(B)) dm= \int h^{-1}_*  \mu^{A,c}_{x}(B) d(h^{-1}_* m)=\int h^{-1}_*  \mu^{A,c}_{x}(B) dm$$
for any measurable set $B$. The uniqueness of Rokhlin decomposition implies \eqref{eq.rokhlin}.\par\vspace{.3em}
Then, $H_{m}(\xi|f\xi)=H_{m}(\eta|A\eta)$ by definition of $\xi$.  
\ep{Claim \ref{metric.entropy.f}}
\begin{claim}\label{claim.equivalent.measures}
$m_{x}^{f,c}$ and $\mu_{x}^{f,c}$ are equivalent $m$-almost every $x\in\T^{3}$ for the partition $\xi$. 
\end{claim}
\bp
$\mu^{f,c}_{x}$ and $m_{x}^{f,c}$ are equivalent if and only if there is an $L^{1}$-function $\rho^{c}_x>0$ so that $\mu^{f,c}_{x}(y)=\rho^{c}_{x}(y).m^{f,c}_{x}(y)$ for $m$-almost $x\in\T^{3}$.\par
A long straightforward computation implies $\rho^{c}_{x}(y)$ has the form (see \cite[Theorem A]{LY1}):

\begin{equation}\label{rhoc}
\rho^c_{x}(y) =  \frac{\Delta^c(x,y)}{ \int_{W_D^{c}(x)} \Delta^c(x,y) dm^{c}_{x}},\qquad \Delta^c(x,y) = \frac{\prod_{i=1}^{+\infty} Jf^{c}(f^{-i}(x) )}{\prod_{i=1}^{+\infty} Jf^{c}(f^{-i}(y) )}. 
\end{equation}
Sicnce $y\mapsto \log\Delta^{c}(x,y)$ is H\"older and $\xi(x)\subset W^{c}_{loc}(x)\subset W^{+}_{loc}(x)$, $\log\Delta^{c}(x,y)$ is 
%uniformly 
bounded from above and away from zero on $\xi(x)$ for $m$-almost $x\in\T^{3}$. 
\ep{Claim \ref{claim.equivalent.measures} + Lemma \ref{lemma.equivalent.measures}}

\begin{claim} \label{jacobian.defined.mae}
The center Jacobian $Jh^{c}$ is defined $m$-almost everywhere. It is measurable, and so is $\log Jh^{c}$.
\end{claim}
\bp
Consider an interval $B$ in $\xi(x)\subset W^{c}_{f,loc}(x)$. Then, 

 $$\mu_{x}^{f,c}(B)=\int_{B}\rho_x dm_{x}^{f,c}=h_*\mu^{f,c}_x(h(B))=\frac1{|\eta_{h(x)}|}m^{A,c}_{h(x)}(h(B)).$$
So $Jh^{c}=|\eta_{h(x)}|.\rho_{x}^{c}$ \; $m$-almost everywhere. This shows that $\log Jh^{c}$ is a measurable function.
\ep{Claim \ref{jacobian.defined.mae}}
\begin{claim} \label{claim.cjacobian.satisfies.cohomological}$\log Jh^{c}$ satisfies the cohomological equation 
\begin{equation}\label{cohomological.equation.c}
\log Jf^{c}(x)-\lambda^{c}_{A}=\phi(x)-\phi(f(x))\quad m{\text -a.e. }x.  
\end{equation}
\end{claim}
\bp
 The center Jacobian of 
 $$h\circ f(x)=Ah(x)$$
is $m$-almost everywhere equal to
$$Jh^{c}(f(x))Jf^{c}(x)=JA^{c}Jh^{c}(x).$$
 \par\vspace{.3em}
 Its $\log$ is
 $$\log Jh^c(f(x))+\log Jf^c(x)=\lambda^c_A+\log Jh^c(x)$$
 and it is defined $m$-almost everywhere.
\ep{Claim \ref{claim.cjacobian.satisfies.cohomological}}
\begin{claim}
\label{claim.wilkinson} 
 Either $f$ is Anosov, or else $Jh^{c}$ coincides $m$-almost everywhere with a continuous function $\phi$ satisfying the cohomological equation \eqref{cohomological.equation.c}. 
\end{claim}
\bp 
If $f$ is not accessible, then $f$ is Anosov, \cite[Corollary 1.1]{ganshi}.\par
Assume that $f$ is accessible. Wilkinson's theorem \cite[Theorem A (III)]{Wilkinson2013} implies $\log Jh^{c}$ coincides $m$-almost everywhere with a continuous function $\phi$ which is a solution of \eqref{cohomological.equation.c}. 
\ep{Claim \ref{claim.wilkinson}}
\begin{claim} \label{claim.f.Anosov}$f|_{E^{c}}$ is uniformly expanding. So, $f$ is Anosov.
\end{claim}
\bp
 $$\lambda^{c}(\mu)=\lambda^{c}_{A}>0,$$ for all ergodic invariant measures $\mu$\label{anosovidad} 
 (Claim \ref{claim.wilkinson}).\par

If $f$ were not Anosov, there would be a unitary vector $v^{c}\in E^{c}$, so that for every $\eps>0$ there is a converging sequence $\{x_{n}\}_{n}$ such that $\|Df^{n}(x_{n})v^{c}\|\leq \eps^{n}$ with $n\to \infty$.
 Any accumulation point of the measures
 $$\mu_{n}=\frac{1}{n}\sum_{k=0}^{n-1}\delta_{f^{k}(x_{n})}$$
 would produce an ergodic invariant measure $\mu$ such that $\lambda^{c}(\mu)\leq 0$, a contradiction.  
 \eqref{anosovidad} implies that $f$ is Anosov.
\ep{Claim \ref{claim.f.Anosov} + Main Theorem.}
\bibliographystyle{alpha}
\bibliography{2025advbibl}
\end{document}